\documentclass[11pt]{article}
\usepackage{mathrsfs}
\usepackage{amsmath}
\usepackage{amssymb}
\usepackage{graphicx}
\usepackage{epic}
\usepackage{color}
\usepackage{cases}
\usepackage{pst-poly}  
\usepackage{pst-plot}  
\usepackage{pst-poly}  
\usepackage{CJK}

\usepackage{verbatim}

\renewcommand{\paragraph}{\roman{paragraph}}

\usepackage{bm}
\usepackage{hyperref}
\usepackage{tikz}
\usetikzlibrary{arrows,shapes,positioning}
\usetikzlibrary{decorations.markings}
\tikzstyle arrowstyle=[scale=1]
\tikzstyle directed=[postaction={decorate,decoration={markings, mark=at position .65 with {\arrow[arrowstyle]{stealth}}}}]
\tikzstyle reverse directed=[postaction={decorate,decoration={markings, mark=at position .65 with {\arrowreversed[arrowstyle]{stealth};}}}]

\newtheorem{theorem}{Theorem}[section]
\newtheorem{corollary}[theorem]{Corollary}

\newtheorem{lemma}[theorem]{Lemma}

\newtheorem{proposition}[theorem]{Proposition}

\newenvironment{proof}{\noindent {\bf Proof.}}{\rule{3mm}{3mm}\par\medskip}

\begin{document}

\title{Counting oriented spanning trees in generalized join  digraphs\thanks{The work was supported by National Natural Science Foundation of China (Grant No. 12271251). \newline $\dag$ Corresponding author.  E-mail addresses: xush0928@163.com (S.  Xu), kexxu1221@126.com
(K.  Xu).}}\author{{Shaohan Xu, Kexiang Xu$^{\dag}$}\\\\{\small School of Mathematics, Nanjing University of Aeronautics and Astronautics,}\\
{\small Nanjing, Jiangsu 211016, PR China}}

\maketitle
\begin{abstract}
Let $G$ be a digraph with vertex set $\{1,2,...,n\}$ and $H_{1},H_{2},...,H_{n}$ be $n$  digraphs.   The generalized join digraph $\overrightarrow{G}=G[H_{1},H_{2},...,H_{n}]$  is a  digraph obtained from $G$ by replacing each vertex $i$ with $H_{i}$ and for any $u\in V(H_{i})$  and $v\in V(H_{j})$, $(u,v)\in E(\overrightarrow{G})$ if and only if $(i,j)\in E(G)$. In this paper we express the number of  oriented spanning trees in $\overrightarrow{G}$  in terms of Laplacian eigenvalues of $H_{1},H_{2},...,H_{n}$ and oriented spanning trees of $G$. Furthermore, we consider the number of oriented spanning trees  with a  fixed root in $\overrightarrow{G}$. First, we   introduce the biclique-directed star  transformation formula for counting oriented spanning trees  with a fixed root in digraphs. Using it, we give the formula for the total number of oriented spanning trees  with  roots in a certain $H_{i}$ $(1\leq i \leq n)$ of $\overrightarrow{G}$ in terms of Laplacian eigenvalues of $H_{1},H_{2},...,H_{n}$ and oriented spanning trees of $G$. As applications, when each $H_{i}$ is  a given digraph, the enumerative formulas for oriented  spanning trees with a fixed root  of $\overrightarrow{G}$   are derived from our work.
\end{abstract}


\section{Introduction}
Let $G$ be a weighted digraph with vertex set $V(G)$, edge set $E(G)$ and each edge $e=(i,j)\in E(G)$ is weighted by an indeterminate  $\omega_{e}(G)=\omega_{ij}(G)$.  The notation $e=(i,j)\in E(G)$ means there exists a directed edge from tail vertex $i=t(e)$ to head vertex $j=h(e)$. The word ``weighted" is omitted when $\omega_{e}(G)=1$ for each $e\in E(G)$.
For $u\in V(G)$, we denote by $N_{G}^{+}(u)=\{v\in V(G):(u,v)\in E(G)\}$ and
$N_{G}^{-}(u)=\{v\in V(G):(v,u)\in E(G)\}$ the set of {\it out-neighbors} and  {\it in-neighbors}, respectively.  The {\it outdegree} and {\it indegree} of a vertex $u$ in $G$ are denoted by $d_{u}^{+}(G)=|N_{G}^{+}(u)|$ and $d_{u}^{-}(G)=|N_{G}^{-}(u)|$, respectively.

An {\it oriented spanning tree} \cite{C2,L1} of a weighted digraph $G$ is a subtree containing all vertices of $G$, in which one vertex, the root, has outdegree $0$, and every other vertex has outdegree $1$.  Let  $\mathcal{T}(G)$ be the set of all oriented spanning trees of $G$ and $\mathcal{T}_{v}(G)$ be the set of oriented spanning trees of $G$ with root $v$. Let $\{\omega_{e}(G)\}_{e\in E(G)}$ be  indeterminates on  $E(G)$. An {\it oriented  spanning tree enumerator} of $G$ is defined as
$$ t(G,\omega)=\sum_{T\in \mathcal{T}(G)} \prod_{e\in E(T)} \omega_{e}(G)$$
and an {\it oriented  spanning tree enumerator rooted at $v$} of $G$ is defined as
$$ t_{v}(G,\omega)=\sum_{T\in \mathcal{T}_{v}(G)} \prod_{e\in E(T)} \omega_{e}(G).$$
 Let $t(G)=|\mathcal{T}(G)|$ denote the number of oriented spanning trees of $G$ and $t_{v}(G)=|\mathcal{T}_{v}(G)|$ denote the number of oriented spanning  trees of $G$ with root $v$. If $\omega_{e}(G)=1$ for each  $e\in E(G)$, then $t(G,\omega)=t(G)$ and $t_{v}(G,\omega)=t_{v}(G)$.

Let $\{\omega_{i}(G)\}_{i\in V(G)}$ be indeterminates on $V(G)$. If each edge $(u,v)\in E(G)$ has weight $\omega_{uv}(G)=\omega_{v}(G)$, then we say that {\it the weights of $G$ are induced by $\{\omega_{i}(G)\}_{i\in V(G)}$}. The weighted {\it line digraph} $\mathcal{L}(G)$ of a weighted digraph $G$ has vertex set $V(\mathcal{L}(G))=E(G)$, and there exists a directed edge from $e$ to $f$ in $\mathcal{L}(G)$ if and only if $h(e)=t(f)$, and  weights of  $\mathcal{L}(G)$ are induced by indeterminates $\{\omega_{e}(G)\}_{e\in E(G)}$.

The problem related to the enumeration of spanning trees in a graph is  an important and popular topic in mathematics, physics, computer science, and so on, and has been studied extensively for a long time. See some relevant results in  \cite{C1,D1,R2,xxx,G23,G1,K1,Y3,Y1,L10,F9}. However, some results about counting oriented spanning trees in   digraphs are sparse. Some scholars have  investigated the enumeration of oriented spanning trees in digraphs and given some nice results.
For line digraphs, Knuth \cite{K2} obtained the  formula for the number of oriented spanning trees with a fixed root. Combinatorial and bijective proofs of Knuth's formula are given in \cite{O1} and \cite{B4}, respectively. In 2011,   Levine  \cite{L1} obtained the following formulas for two oriented spanning tree enumerators in weighted line digraphs, which is a generalization of Knuth's result. Independently, Sato \cite{S1} presented a new proof for two of Levine's theorems by using the generalized Matrix-Tree theorem.

\begin{theorem}{\rm (\cite{L1})}\label{x1}
Let $G$ be a weighted digraph such that $d_{u}^{-}(G)>0$ for each $u\in V(G)$. Then
$$t(\mathcal{L}(G),\omega)=t(G,\omega)\prod_{v\in V(G)}d_{v}(G)^{d_{v}^{-}(G)-1},$$
where  $d_{v}(G)=\sum_{t(e)=v}\omega_{e}(G)$.
\end{theorem}

\begin{theorem}{\rm (\cite{L1})}\label{x2}
Let $G$ be a weighted digraph such that $d_{u}^{-}(G)>0$ for each $u\in V(G)$.  For any edge $e^{*}=(u^{*},v^{*})$ of $G$ satisfying $d_{v^{*}}^{-}(G)\geq2$, we have
$$t_{e^{*}}(\mathcal{L}(G),\omega)=\omega_{e^{*}}(G)t_{u^{*}}(G,\omega)d_{v^{*}}(G)^{d_{v^{*}}^{-}(G)-2}\prod_{v^{*}\neq v\in V(G)}d_{v}(G)^{d_{v}^{-}(G)-1}.$$
\end{theorem}

It is worth noting that Xu, Zhang, Ning and Li  \cite{C4} proved that the above two formulas also hold when the condition $d_{u}^{-}(G)>0$ is removed.

In 2022 Chen, Jin and  Yan  \cite{C5} derived the enumerative  formulas for two oriented spanning tree enumerators  of a weighted middle digraph. In 2023 Zhou and Bu \cite{Z2}, using biclique partitions of digraphs, obtained reduction formulas for the number of oriented spanning trees with a fixed root of digraphs, which extend the results of Knuth and Levine from  line digraphs   to general digraphs.

Let $G$ be a digraph with vertex set $\{1,2,...,n\}$ and $H_{1},H_{2},...,H_{n}$ be $n$  digraphs.   The {\it generalized join digraph} $\overrightarrow{G}=G[H_{1},H_{2},...,H_{n}]$ of $H_{1},H_{2},...,H_{n}$ with respect to $G$ is  a  digraph obtained from $G$ by replacing each vertex $i$ with $H_{i}$ and for any $u\in V(H_{i})$  and $v\in V(H_{j})$, $(u,v)\in E(\overrightarrow{G})$ if and only if $(i,j)\in E(G)$.  In 2021, Zhou and Bu \cite{Z1} used the Schur complement formula to obtain the local complement transformation for enumerating spanning trees in undirected graphs, and gave  an expression for counting spanning trees in  the generalized join graph $G[H_{1},H_{2},...,H_{n}]$ of an undirected graph $G$ in terms of Laplacian eigenvalues of $H_{1},H_{2},...,H_{n}$ and spanning trees of $G$, see Theorem $6.22$ of \cite{Z1}.

In this paper, motivated by the above result about undirected graphs, we study the enumeration of oriented spanning trees (with a fixed root) in  generalized join digraphs. In Section $2$ we  introduce the biclique-directed star  transformation formula for enumerating oriented spanning trees with a fixed root in digraphs. In Section $3$  we provide an expression of the characteristic polynomial of Laplacian matrix for a generalized join digraph $\overrightarrow{G}$. Furthermore, we express the number of  oriented spanning trees in  $\overrightarrow{G}$  in terms of Laplacian eigenvalues of $H_{1},H_{2},...,H_{n}$ and oriented spanning trees of $G$. In Section $4$,  using the biclique-directed star  transformation formula, we give the formula for the total number of  oriented spanning trees  with  roots in a certain  $H_{i}$ $(1\leq i \leq n)$ of $\overrightarrow{G}$ in terms of Laplacian eigenvalues of $H_{1},H_{2},...,H_{n}$ and oriented spanning trees of $G$. As applications, when each $H_{i}$ is  a given digraph, the enumerative formulae for oriented  spanning trees with a fixed root  of $\overrightarrow{G}$   are derived from our work.

\section{Preliminaries}
\subsection{Laplacian matrix of weighted digraphs}
Let $G$ be a weighted digraph on $n$ vertices, and each edge $e=(i,j)\in E(G)$ is weighted by an indeterminate $\omega_{ij}(G)$. The weighted degree of a vertex $i$ is $d_{i}(G)=\sum_{(i,j)\in E(G)}\omega_{ij}(G)$. The {\it Laplacian matrix} $L_{G}$ of $G$ is an $n\times n$ matrix with entries
\begin{displaymath}
(L_{G})_{ij}=
   \begin{cases}
   d_{i}(G) &\mbox {\rm if $i=j$},\\
 -\omega_{ij}(G) &\mbox {\rm if $(i, j)\in E(G)$},\\
 0  &\mbox {\rm otherwise}.
   \end{cases}
\end{displaymath}

Let  $A(i,j)$ denote the submatrix of a matrix $A$ obtained by deleting the $i$-th row and $j$-th column,  and let $det(A)$ denote the determinant of a square matrix $A$. The following result holds from the all minors matrix tree theorem \cite{C3}.

\begin{lemma}{\rm (\cite{C3})}\label{x7}
Let $G$ be a weighted digraph with $V(G)=\{1,2,...,n\}$. For any $i,j\in V(G)$, we have
$$det(L_{G}(i,j))=(-1)^{i+j}t_{i}(G,\omega).$$
 Furthermore, $t(G,\omega)=tr(adj(L_{G}))=\sum_{i\in V(G)}t_{i}(G,\omega)$,
 where $adj(A)$ and $tr(A)$ are the adjoint matrix  and the trace of a square matrix $A$, respectively.
\end{lemma}

For a matrix $M$, let $M[i_{1},i_{2},...,i_{s}|j_{1},j_{2},...,j_{t}]$ denote an $s\times t$ submatrix of $M$ whose row indices and column indices are $i_{1},i_{2},...,i_{s} $ and $j_{1},j_{2},...,j_{t}$, respectively. The following is a determinant identity involving the Schur complement.
\begin{lemma}{\rm (\cite{B3})}\label{x8}
Let $M=\left(\begin{matrix}
A &B\\
C &D\\
\end{matrix}
\right)$ be a block matrix of order $n$ where $A=M[1,2,...,k| $ $1,2,...,k]$ is nonsingular. If $k+1\leq i_{1}< i_{2}<\cdots <i_{s}\leq n$ and $k+1\leq j_{1}< j_{2}<\cdots <j_{s}\leq n$, then
$$\frac{det(M[1,2,...,k,i_{1},i_{2},...,i_{s}|1,2,...,k,j_{1},j_{2},...,j_{s}])}{det(A)}=det(S[i_{1},i_{2},...,i_{s}|j_{1},j_{2},...,j_{s}]),$$
where $S=D-CA^{-1}B$.
\end{lemma}

Let $f(A,\lambda)=det(\lambda I_{n}-A)$ be  the {\it characteristic polynomial} of a matrix $A$ of order $n$, where $I_{n}$ is the identity matrix of order $n$, and let $Spec(A)$ denote the {\it spectrum} of the matrix $A$. For a weighted digraph $G$ of order $n$, the {\it Laplacian eigenvalues} $\mu_{i}(G)$ for $1\leq i\leq n$ of $G$ are the roots of the  characteristic polynomial of $L_{G}$ and the {\it Laplacian spectrum} $Spec_{L}(G)$ of $G$ is the set of all  Laplacian eigenvalues of $G$. Let  $\mathrm{j}_{n}$ denote the all-ones column vector of dimension $n$.  Since all row  sums of $L_{G}$ are $0$, we have  $L_{G}\mathrm{j}_{n}=0\mathrm{j}_{n}$, that is, $L_{G}$ has an eigenvalue $0$ with corresponding eigenvector $\mathrm{j}_{n}$. In this paper  we will assume that the last Laplacian eigenvalue of $G$ is $0$, that is,  $\mu_{n}(G)=0$.

\subsection{Biclique-directed star  transformation formulas}
Let $G$ be a weighted digraph with a partition $V(G)=V_{1}\cup V_{2}$ and
$L_{G}=\left(\begin{matrix}
L_{1} &B\\
C &L_{2}\\
\end{matrix}
\right)$
where  $L_{1}$ and $L_{2}$ are principal submatrices of $L_{G}$ corresponding to $V_{1}$ and $V_{2}$, respectively. If $L_{1}$ and $L_{2}$ are nonsingular, then the Schur complement of   $L_1$ and   $L_{2}$ in  $L_{G}$ are   $S_{2}=L_{2}-CL_{1}^{-1}B$ and $S_{1}=L_{1}-BL_{2}^{-1}C$, respectively. Note that $S_{1}$ and $S_{2}$ are the Laplacian matrices of some weighted digraphs since all row sums are zero in $S_{1}$ and $S_{2}$. We define the weighted digraph with Laplacian matrix $S_{k}$ as the {\it  Schur complement weighted digraph} of $G$ with respect to $V_{k}$, denoted by $G(V_{k})$ ($k=1,2$).  Then $G(V_{k})$ has vertex set $V_{k}$ and edge set $\{(u,v):(S_{k})_{uv}\neq 0, u\neq v\} $.  The following is the Schur complement formula of $t_{u}(G,\omega)$ for a weighted digraph $G$.

\begin{proposition}\label{x9}
Let $G$ be a weighted digraph with a partition $V(G)=V_{1}\cup V_{2}$, and   $L_{1}$ and $L_{2}$ be principal submatrices of $L_{G}$ corresponding to $V_{1}$ and $V_{2}$, respectively.  Set $i\neq j\in \{1,2\}$. If $L_{i}$ is nonsingular and $u\in V_{j}$, then
$$t_{u}(G,\omega)=det(L_{i})t_{u}(G(V_{j}),\omega).$$
\end{proposition}
\begin{proof}
Without loss of generality, let $i=1$ and $j=2$. The Laplacian matrix of $G$ can be written as $L_{G}=\left(\begin{matrix}
L_{1} &B\\
C &L_{2}\\
\end{matrix}
\right)$. The Schur complement $L_{2}-CL_{1}^{-1}B=L_{G(V_{2})}$ is the Laplacian matrix of the Schur complement weighted digraph $G(V_{2})$. By Lemmas \ref{x7} and \ref{x8}, we have
$$t_{u}(G,\omega)=det(L_{1})det(L_{G(V_{2})}(u,u))=det(L_{1})t_{u}(G(V_{2}),\omega),~{\rm for~}u\in V_{2},$$
where $L_{G(V_{2})}(u,u)$ is the submatrix of $L_{G(V_{2})}$ obtained by deleting the $u$-th row and the $u$-th column.
{\hfill}
\end{proof}

A {\it biclique} \cite{G2}  is a complete bipartite digraph $Q$ whose vertices can be partitioned into two parts $Q^{(1)}$ and $Q^{(2)}$, and $E(Q)=\{(i,j):i\in Q^{(1)}, j\in Q^{(2)} \}$. Let $G$ be a weighted digraph with weights  induced by indeterminates $\{\omega_{i}(G)\}_{i\in V(G)}$ which  contains a weighted  biclique $Q$ as a weighted sub-digraph.  We define a weighted directed star $K_{1,|Q|}^{\omega}$  with vertex set $\{o\}\cup Q^{(1)}\cup Q^{(2)}$ and edge set $\{(u,o):u\in Q^{(1)}\}\cup \{(o,v):v\in Q^{(2)}\}$, where the weight of each edge in $K_{1,|Q|}^{\omega}$ is defined as
\begin{displaymath}\tag{2.1}
   \begin{cases}
 \omega_{uo}(K_{1,|Q|}^{\omega})=\sum_{v\in Q^{(2)}}\omega_{v}(G) &\mbox{if  $u\in Q^{(1)}$},\\
  \omega_{ov}(K_{1,|Q|}^{\omega})=\omega_{v}(G) &\mbox{if $v\in Q^{(2)}$}.\\
   \end{cases}
\end{displaymath}

By using the Schur complement formula in Proposition \ref{x9}, we can obtain the following biclique-directed star transformation formula for $t_{u}(G,\omega)$ as a useful tool for  counting  oriented spanning trees with a fixed root of generalized join digraphs in  Section $4$.
\begin{lemma}\label{x11}
Let $G$ be a weighted digraph, containing a weighted  biclique $Q$ as a weighted sub-digraph, whose weights are induced by indeterminates $\{\omega_{i}(G)\}_{i\in V(G)}$. Assume that $G^{*}$ is the weighted digraph  obtained from $G$ by replacing $Q$ with the weighted directed star $K_{1,|Q|}^{\omega}$ defined in $(2.1)$. Set $\omega(Q)=\sum_{v\in Q^{(2)}}\omega_{v}(G)\neq0$. Then
\begin{displaymath}
\begin{split}
&t_{u}(G^{*},\omega)=\omega(Q) t_{u}(G,\omega),~{\rm for~any}~u\in V(G),\\
&t_{o}(G^{*},\omega)=\sum_{v\in Q^{(1)}}t_{v}(G^{*},\omega)=\omega(Q)\sum_{v\in Q^{(1)}}t_{v}(G,\omega).\\
 \end{split}
 \end{displaymath}
\end{lemma}
\begin{proof}
The Laplacian matrix of $G^{*}$ can be written as $L_{G^{*}}=\left(\begin{matrix}
d_{o}(G^{*}) &x^{\top}\\
y &L_{2}\\
\end{matrix}
\right),$
where $L_{2}$ is the principal submatrix of $L_{G^{*}}$ corresponding to $V(G)$, and $x,y$ are two column vectors such that
\begin{displaymath}
  x_{j}= \begin{cases}
-\omega_{j}(G) &\mbox{if  $j\in Q^{(2)}$},\\
0 &\mbox{otherwise};\\
   \end{cases}~~~~~~~~~~y_{j}= \begin{cases}
-\omega(Q) &\mbox{if  $j\in Q^{(1)}$},\\
0 &\mbox{otherwise}.\\
   \end{cases}
\end{displaymath}

 By the definition of $G^{*}$, for two distinct vertices $i,j\in V(G)$, we have
\begin{displaymath}
  (L_{2})_{ij}= \begin{cases}
0 &\mbox{if  $(i,j)\in E(Q)$},\\
-\omega_{j}(G) &\mbox{if $(i,j)\in E(G) \setminus E(Q)$},\\
0 &\mbox{if $(i,j)\notin E(G)$},\\
 \end{cases}
\end{displaymath}
and the weighted degree of $o$ in $G^{*}$ is
$$d_{o}(G^{*})=\sum_{v\in Q^{(2)}}\omega_{v}(G)=\omega(Q)\neq 0.$$
The Schur complement weighted digraph $G^{*}(V(G))$ has Laplacian matrix $L_{G^{*}(V(G))}=L_{2}-d_{o}(G^{*})^{-1}yx^{\top}$. For any two distinct vertices $i,j\in V(G)$, we have
\begin{displaymath}
  (yx^{\top})_{ij}=\begin{cases}
\omega_{j}(G)\omega(Q) &\mbox{if  $(i,j)\in E(Q)$},\\
0 &\mbox{if $(i,j)\in E(G) \setminus E(Q)$},\\
0 &\mbox{if $(i,j)\notin E(G)$},\\
 \end{cases}
\end{displaymath}
and
\begin{displaymath}
\left(L_{2}-d_{o}(G^{*})^{-1}yx^{\top}\right)_{ij}
=\begin{cases}
-\omega_{j}(G) &\mbox{if $(i,j)\in E(G)$},\\
0 &\mbox{if $(i,j)\notin E(G)$},
\end{cases}
=(L_G)_{ij}.
\end{displaymath}
Since both $L_{2}-d_{o}(G^{*})^{-1}yx^{\top}$ and $L_G$ have zero row sums, their diagonal entries also coincide.
Hence, $L_{G^{*}(V(G))}=L_{G}$, that is, $G^{*}(V(G))$ and $G$ are isomorphic weighted digraphs. By Proposition \ref{x9}, we have
\begin{displaymath}\tag{2.2}
t_{u}(G^{*},\omega)=d_{o}(G^{*})t_{u}(G^{*}(V(G)),\omega)=\omega(Q)t_{u}(G,\omega).
\end{displaymath}

Let $X$ be the adjoint matrix of $L_{G^{*}}$. Then
\begin{displaymath}\tag{2.3}
 XL_{G^{*}}=det(L_{G^{*}})I=0.
\end{displaymath}
By Lemma \ref{x7}, we get
\begin{displaymath}\tag{2.4}
 (X)_{uv}=t_{v}(G^{*},\omega),~{\rm for~any~}u,v\in V(G^{*}).
\end{displaymath}
Hence,  by Equations (2.3) and (2.4), we have
\begin{displaymath}
 \sum_{v\in V(G^{*})}t_{v}(G^{*},\omega)(L_{G^{*}})_{vo}=d_{o}(G^{*})t_{o}(G^{*},\omega)-\sum_{v\in Q^{(1)}}t_{v}(G^{*},\omega)\omega(Q)=0.
\end{displaymath}
By Equation (2.2), we get
$$t_{o}(G^{*},\omega)=\sum_{v\in Q^{(1)}}t_{v}(G^{*},\omega)=\omega(Q)\sum_{v\in Q^{(1)}}t_{v}(G,\omega).$$
{\hfill}
\end{proof}

Taking $\omega_{v}(G)=1$ for all $v \in V(G)$,   we obtain a  weighted directed star $K_{1,|Q|}$ with vertex set $\{o\}\cup Q^{(1)}\cup Q^{(2)}$ and edge set $\{(u,o):u\in Q^{(1)}\}\cup \{(o,v):v\in Q^{(2)}\}$, where the weight of each edge in $K_{1,|Q|}$ is defined as
\begin{displaymath}\tag{2.5}
   \begin{cases}
 \omega_{uo}(K_{1,|Q|})=|Q^{(2)}| &\mbox{if  $u\in Q^{(1)}$},\\
  \omega_{ov}(K_{1,|Q|})=1 &\mbox{if $v\in Q^{(2)}$}.\\
   \end{cases}
\end{displaymath}
Now we  give  the following result from Lemma \ref{x11}.

\begin{lemma}\label{xsh-1}
Let $G$ be a  digraph  containing a   biclique $Q$ as a  sub-digraph. Assume that $G^{*}$ is the weighted digraph  obtained from $G$ by replacing $Q$ with the weighted directed star $K_{1,|Q|}$ defined in $(2.5)$. Then
\begin{displaymath}
\begin{split}
&t_{u}(G^{*},\omega)=|Q^{(2)}| t_{u}(G),~{\rm for~any}~u\in V(G),\\
&t_{o}(G^{*},\omega)=\sum_{v\in Q^{(1)}}t_{v}(G^{*},\omega)=|Q^{(2)}|\sum_{v\in Q^{(1)}}t_{v}(G).\\
 \end{split}
 \end{displaymath}
\end{lemma}

\section{Counting oriented spanning trees  in $G[H_{1},H_{2},...,H_{n}]$}
 Let $M$  be a complex matrix of order $n$ described in the following block form

\begin{displaymath}\tag{3.1}
M=\left(\begin{matrix}
M_{11} &M_{12} &\cdots &M_{1t}\\
M_{21} &M_{22} &\cdots &M_{2t}\\
  \vdots &\vdots &\ddots &\vdots\\
  M_{t1} &M_{t2} &\cdots &M_{tt}
   \end{matrix}
\right),
\end{displaymath}
where the blocks $M_{ij}$ are $n_{i}\times n_{j}$ matrices for any $1\leq i,j\leq t$ and $n=n_{1}+ n_{2}+ \cdots +n_{t}$. For $1\leq i,j\leq t$, let $b_{ij}$ denote the average row sum of $M_{ij}$. The quotient matrix $B(M)= (b_{ij})$
is a $t\times t$ matrix whose entries are the average row sums of the blocks $M_{ij}$ of $M$. If each block $M_{ij}$ of $M$ has constant row sum, i.e., $M_{ij}\mathrm{j}_{n_{j}}=b_{ij}\mathrm{j}_{n_{i}}$, then the matrix $B(M)$ is called the {\it equitable quotient matrix} of $M$. The following observation about the relation between the spectrum of a  complex matrix and its equitable quotient matrix can be found in \cite{Y2}.

\begin{lemma}\label{x211}{\rm (\cite{Y2})}
The eigenvalues of the equitable quotient matrix $B(M)$ are the eigenvalues of the matrix $M$, where $M$ is the matrix given by $(3.1)$.
\end{lemma}

Let $A$ be an $n\times n$ matrix. The conjugate transpose matrix $A^{*}$ of $A$ is defined by $A^{*}=\overline{A}^{\top}$, where $\overline{A}$ is a matrix obtained  by taking the conjugate of each element of $A$. $A$ is called a unitary matrix if $AA^{*}=A^{*}A=I_{n}$. Let  $J_{p,q}$ denote the $p\times q$ matrix whose every entry is $1$.

\begin{lemma}\label{x2113}{\rm (\cite{H1}, Schur's unitary triangularization theorem)}
Given an $n\times n$ matrix $A$ with eigenvalues $\lambda_{1},\lambda_{2},\ldots,\lambda_{n}$ in any prescribed order, there is a unitary matrix $U$ such that
$$U^{*}AU=Y=(y_{ij})$$
is an upper triangular matrix with diagonal entries $y_{ii}=\lambda_{i}$ for each $i=1,2,...,n$. That is, every square matrix $A$ is unitarily similar to a triangular matrix whose diagonal entries are the eigenvalues of $A$ in a prescribed order. Moreover, the first column of   $U$ can be an eigenvector of $A$ corresponding to the eigenvalue $\lambda_{1}$.
\end{lemma}

\begin{lemma}\label{x2112}
Let $M$ be a matrix  defined in $(3.1)$ such that $M_{ij}=s_{ij}J_{n_{i},n_{j}}$ for $i\neq j$, and each $M_{ii}$ be a square matrix with constant row sum $p_{ii}$. Then the equitable quotient matrix of $M$ is $B(M)=(b_{ij})$ with  $b_{ij}=s_{ij}n_{j}$ if $i\neq j$, and $b_{ii}=p_{ii}$. Moreover,
$$Spec(M)=Spec(B(M))\cup \bigcup_{i=1}^{t}(Spec(M_{ii})\setminus \{p_{ii}\}).$$
\end{lemma}
\begin{proof}
Since $M_{ii}\mathrm{j}_{n_{i}}=p_{ii}\mathrm{j}_{n_{i}}$ for $1\leq i\leq t$, we let $Spec(M_{ii})=\{p_{ii},\lambda_{2}^{(i)},...,\lambda_{n_{i}}^{(i)}\}$.    By Lemma \ref{x2113},  for any $1\leq i\leq t$, there is an $n_{i}\times n_{i}$ unitary matrix $U_{n_{i}}$ with its first column being $\frac{1}{\sqrt{n_{i}}}\mathrm{j}_{n_{i}}$ such that $$U_{n_{i}}^{*}M_{ii}U_{n_{i}}=Y_{n_{i}}=(y_{ab}),$$
 where $Y_{n_{i}}$ is upper triangular with diagonal entries $y_{11}=p_{ii}$ and $y_{kk}=\lambda_{k}^{(i)}$ for each $k=2,3,...,n_{i}$. In particular, for $i\neq j$, we have
  $$U_{n_{i}}^{*}M_{ij}U_{n_{j}}=U_{n_{i}}^{*}(s_{ij}J_{n_{i},n_{j}})U_{n_{j}}=s_{ij}\sqrt{n_{i}n_{j}}E_{n_{i},n_{j}},$$
where $E_{n_{i},n_{j}}$ is an $n_{i}\times n_{j}$ matrix with all zero entries except the (1,1)-entry equals $1$.

Now let $U=diag(U_{n_{1}},U_{n_{2}},...,U_{n_{t}})$ be a block diagonal matrix. Clearly, $U$ is a unitary matrix. Moreover,  $U^{*}MU$ has the block form with
$$(U^{*}MU)_{ii}=Y_{n_{i}}=\left(\begin{matrix}
p_{ii} &*\\
0 &Y_{n_{i}}^{*}\\
\end{matrix}
\right),$$
where $Y_{n_{i}}^{*}$ is the principal submatrix of $Y_{n_{i}}$ in an    upper triangular  form with diagonal entries $y_{kk}=\lambda_{k}^{(i)}$ for $k=2,3,...,n_{i}$ and for $i\neq j$,
$$(U^{*}MU)_{ij}=s_{ij}\sqrt{n_{i}n_{j}}E_{n_{i},n_{j}}.$$

Therefore, $U^{*}MU$ is permutationally similar to a block  upper triangular matrix  with diagonal blocks $B',Y_{n_{1}}^{*},Y_{n_{2}}^{*},...,Y_{n_{t}}^{*}$, where
\begin{displaymath}
B'=\left(\begin{matrix}
p_{11} &s_{12}\sqrt{n_{1}n_{2}} &\cdots &s_{1t}\sqrt{n_{1}n_{t}}\\
s_{21}\sqrt{n_{2}n_{1}} &p_{22} &\cdots &s_{2t}\sqrt{n_{2}n_{t}}\\
  \vdots &\vdots &\ddots &\vdots\\
 s_{t1}\sqrt{n_{t}n_{1}}  & s_{t2}\sqrt{n_{t}n_{2}} &\cdots &p_{tt}
   \end{matrix}
\right).
\end{displaymath}
Note that $B(M)=D^{-1}B'D$, where $D=diag(\sqrt{n_{1}},\sqrt{n_{2}},...,\sqrt{n_{t}})$, and so $Spec(B')=Spec(B(M))$. Finally,
\begin{displaymath}
\begin{split}
Spec(M)&=Spec(U^{*}MU)=Spec(B')\cup\bigcup_{i=1}^{t}(Spec(M_{ii})\setminus \{p_{ii}\})\\
&=Spec(B(M))\cup \bigcup_{i=1}^{t}(Spec(M_{ii})\setminus \{p_{ii}\}).
\end{split}
\end{displaymath}
{\hfill}
\end{proof}

In the following, we express the characteristic polynomial of Laplacian matrix of $\overrightarrow{G}=G[H_{1},H_{2},...,H_{n}]$ in terms of  Laplacian eigenvalues of $H_{1},H_{2},...,H_{n}$ and the equitable quotient matrix of Laplacian matrix of $\overrightarrow{G}$.
\begin{lemma}\label{x13145}
Let $\overrightarrow{G}=G[H_{1},H_{2},...,H_{n}]$  be a  generalized join digraph with $m_{i}=|V(H_{i})|$ for $1\leq i\leq n$.  Then the characteristic polynomial of $L_{\overrightarrow{G}}$  can be expressed by
\begin{displaymath}
\begin{split}
f(L_{\overrightarrow{G}},\lambda)=
f(B(L_{\overrightarrow{G}}),\lambda)\prod_{i=1}^{n}\prod_{k=1}^{m_{i}-1}\left(\lambda-\mu_{k}(H_{i})-\sum_{(i,j)\in E(G)}m_{j}\right),
  \end{split}
\end{displaymath}
where
\begin{displaymath}\tag{3.2}
B(L_{\overrightarrow{G}})=\left(\begin{matrix}
\sum_{(1,j)\in E(G)}m_{j} &-a_{12} &\cdots &-a_{1n}\\
-a_{21} &\sum_{(2,j)\in E(G)}m_{j} &\cdots &-a_{2n}\\
  \vdots &\vdots &\ddots &\vdots\\
  -a_{n1} &-a_{n2} &\cdots &\sum_{(n,j)\in E(G)}m_{j}
   \end{matrix}
\right),
\end{displaymath}
 where  $a_{ij}=m_{j}$ if $(i,j)\in E(G)$, $a_{ij}=0$ otherwise.
\end{lemma}
\begin{proof}
The Laplacian matrix of $\overrightarrow{G}$ can be partitioned as
\begin{displaymath}
L_{\overrightarrow{G}}=\left(\begin{matrix}
\Lambda_{1} &-A_{m_{1}\times m_{2}} &\cdots &-A_{m_{1}\times m_{n}}\\
-A_{m_{2}\times m_{1}} &\Lambda_{2} &\cdots &-A_{m_{2}\times m_{n}}\\
  \vdots &\vdots &\ddots &\vdots\\
  -A_{m_{n}\times m_{1}} &-A_{m_{n}\times m_{2}} &\cdots &\Lambda_{n}
   \end{matrix}
\right),
\end{displaymath}
where
 \begin{displaymath}
\begin{split}
 \Lambda_{i}=L_{H_{i}}+\sum_{(i,j)\in E(G)}m_{j}I,~{\rm for}~i=1,2,...,n,
  \end{split}
\end{displaymath}
and $A_{m_{i}\times m_{j}}=J_{m_{i},m_{j}}$ if $(i,j)\in E(G)$; $A_{m_{i}\times m_{j}}=\mathbf{0}_{m_{i}, m_{j}}$ otherwise, where $\mathbf{0}_{m_{i}, m_{j}}$ is the $m_{i}\times m_{j}$ zero matrix.

  Note that the equitable quotient matrix of   $L_{\overrightarrow{G}}$ can be seen in Equation (3.2). By Lemma \ref{x2112}, we have
$$Spec(L_{\overrightarrow{G}})=Spec(B(L_{\overrightarrow{G}}))\cup \bigcup_{i=1}^{n}\left(Spec(\Lambda_{i})\setminus \left\{\sum\nolimits_{(i,j)\in E(G)}m_{j}\right\}\right).$$
Therefore, \begin{displaymath}
\begin{split}
f(L_{\overrightarrow{G}},\lambda)&=
f(B(L_{\overrightarrow{G}}),\lambda)\prod_{i=1}^{n}\frac{f(\Lambda_{i},\lambda)}{(\lambda-\sum_{(i,j)\in E(G)}m_{j})}\\
&=f(B(L_{\overrightarrow{G}}),\lambda)\prod_{i=1}^{n}\prod_{k=1}^{m_{i}-1}\left(\lambda-\mu_{k}(H_{i})-\sum_{(i,j)\in E(G)}m_{j}\right).
  \end{split}
\end{displaymath}
{\hfill}
\end{proof}

 Now we   give an expression for enumerating  oriented spanning trees of $\overrightarrow{G}$  in terms of Laplacian eigenvalues of $H_{1},H_{2},...,H_{n}$ and oriented spanning trees of $G$ as follows.

\begin{theorem}\label{x139}
Let $\overrightarrow{G}=G[H_{1},H_{2},...,H_{n}]$  be a  generalized join digraph.  Then
\begin{displaymath}
\begin{split}
 t(\overrightarrow{G})
 =&\prod_{i\in V(G)}\prod_{k=1}^{m_{i}-1}\left(\mu_{k}(H_{i})+\sum_{(i,j)\in E(G)}m_{j}\right)\sum_{T\in \mathcal{T}(G)}\prod_{(i,j)\in E(T)}m_{j},
  \end{split}
\end{displaymath}
where $m_{i}=|V(H_{i})|$.
\end{theorem}
\begin{proof}
For the characteristic polynomial $f(A,\lambda)=|\lambda I_{n}-A|=\lambda^{n}+a_{n-1}\lambda^{n-1}+\cdots +a_{1}\lambda+a_{0}$, the coefficient of $\lambda$ is $(-1)^{n-1}tr(adj(A))$.

Note that  $B(L_{\overrightarrow{G}})$  defined in Lemma \ref{x13145} is the Laplacian  matrix of  a weighted digraph $G'$ weighting on $G$, and   weights of $G'$ are induced by indeterminates $\{m_{i}\}_{i\in V(G)}$. Set $N=\sum_{i=1}^{n}m_{i}$.  Therefore, by Lemma \ref{x7}, we have the two derivatives
\begin{displaymath}\tag{3.3}
\begin{split}
f'(L_{\overrightarrow{G}},0)=(-1)^{N-1}t(\overrightarrow{G}),
\end{split}
\end{displaymath}
\begin{displaymath}\tag{3.4}
\begin{split}
f'(B(L_{\overrightarrow{G}}),0)=(-1)^{n-1}\sum_{T\in \mathcal{T}(G)}\prod_{(i,j)\in E(T)}m_{j}.
\end{split}
\end{displaymath}
Since  $f(B(L_{\overrightarrow{G}}),0)=0$, by Lemma \ref{x13145}, we have
\begin{displaymath}
\begin{split}
f'(L_{\overrightarrow{G}},0)=(-1)^{N-n}\prod_{i\in V(G)}\prod_{k=1}^{m_{i}-1}\left(\mu_{k}(H_{i})+\sum_{(i,j)\in E(G)}m_{j}\right)f'(B(L_{\overrightarrow{G}}),0).
\end{split}
\end{displaymath}
Hence, from  Equations (3.3) and (3.4),  the result follows.
{\hfill}
\end{proof}

If each $|V(H_{i})|=m_{i}=m$ for $1\leq i\leq n$ in Theorem \ref{x139}, we obtain the following result.

\begin{corollary}\label{x123x}
Let $\overrightarrow{G}=G[H_{1},H_{2},...,H_{n}]$  be a  generalized join digraph with $|V(H_{i})|=m$ for $1\leq i\leq n$. Then
\begin{displaymath}
\begin{split}
t(\overrightarrow{G})
 =&m^{n-1}t(G)\prod_{i\in V(G)}\prod_{k=1}^{m-1}\left(\mu_{k}(H_{i})+md_{i}^{+}(G)\right).
  \end{split}
\end{displaymath}
\end{corollary}

\section{Counting oriented spanning trees  rooted in $G[H_{1},H_{2},...,H_{n}]$}

 In this section we consider    the number of oriented spanning trees  with a  fixed root in generalized join digraphs.
 Before giving main results, we shall introduce the following theorem.
\begin{theorem}{\rm (\cite{C4})}\label{xb}
Let $G$ be a weighted digraph with $\omega_{e}(G)>0$ for each $e\in E(G)$, and let $e^{*}=(u^{*},v^{*})$ be an edge of  $G$.

{\em (1)} If $d_{v^{*}}^{+}(G)=0$, then
$$t_{e^{*}}(\mathcal{L}(G),\omega)=\omega_{e^{*}}(G)t_{u^{*}}(G,\omega)\prod_{v^{*}\neq v\in V(G)}d_{v}(G)^{d_{v}^{-}(G)-1}.$$

{\em (2)} If $d_{v^{*}}^{+}(G)>0$, then
$$t_{e^{*}}(\mathcal{L}(G),\omega)=\omega_{e^{*}}(G)t_{u^{*}}(G,\omega)d_{v^{*}}(G)^{d_{v^{*}}^{-}(G)-2}\prod_{v^{*}\neq v\in V(G)}d_{v}(G)^{d_{v}^{-}(G)-1}.$$
\end{theorem}

Now we  give an expression of $\sum_{u^{*}\in V(H_{i^{*}})}t_{u^{*}}(\overrightarrow{G})$ for  $i^{*}\in V(G)$ in terms of Laplacian eigenvalues of $H_{1},H_{2},...,H_{n}$ and oriented spanning trees of $G$ as follows.
\begin{theorem}\label{x13}
Let $\overrightarrow{G}=G[H_{1},H_{2},...,H_{n}]$  be a  generalized join digraph with $|V(H_{i})|=m_{i}$ for $1\leq i\leq n$ and $i^{*}\in V(G)$. Then
\begin{displaymath}
\begin{split}
\sum_{u^{*}\in V(H_{i^{*}})} t_{u^{*}}(\overrightarrow{G})
 =&\prod_{i\in V(G)}\prod_{k=1}^{m_{i}-1}\left(\mu_{k}(H_{i})+\sum_{(i,j)\in E(G)}m_{j}\right)\sum_{T\in \mathcal{T}_{i^{*}}(G)}\prod_{(i,j)\in E(T)}m_{j}.
  \end{split}
\end{displaymath}

\end{theorem}
\begin{proof}
 First, we consider that $d_{i}^{+}(G)>0$ for each $i\in V(G)$. For each $e=(i,j)\in E(G)$, $\overrightarrow{G}$ contains a  biclique $Q_{e}$ as a sub-digraph with vertex set $V(H_{i})\cup V(H_{j})$. By Lemma \ref{xsh-1}, we  can  replace each  biclique $Q_{e}$ in $\overrightarrow{G}$ with  a weighted directed star $K_{1,|Q_{e}|}$  with vertex set $\{e\}\cup V(H_{i})\cup  V(H_{j})$, edge set $\{(u,e):u\in V(H_{i})\}\cup \{(e,v):v\in V(H_{j})\}$ and the weight of every edge of  $K_{1,|Q_{e}|}$  is defined in Equation (2.5).
 Then we obtain a weighted digraph $\overrightarrow{G}'$  with vertex set $E(G)\cup V(H_{1})\cup\cdots \cup V(H_{n})$. By Lemma \ref{xsh-1},  for  two vertices $u^{*}\in V(H_{i^{*}})$  and  $e^{*}=(i^{*},j^{*})\in E(G)\subseteq V(\overrightarrow{G}')$ in $\overrightarrow{G}'$, we have
 \begin{displaymath}
\begin{split}
 t_{u^{*}}(\overrightarrow{G}',\omega)=\prod_{e\in E(G)}|Q_{e}^{(2)}| t_{u^{*}}(\overrightarrow{G}),
  \end{split}
\end{displaymath}
 \begin{displaymath}\tag{4.1}
\begin{split}
 t_{e^{*}}(\overrightarrow{G}',\omega)=\sum_{u^{*}\in V(H_{i^{*}})}t_{u^{*}}(\overrightarrow{G}',\omega)=\prod_{e\in E(G)}|Q_{e}^{(2)}|\sum_{u^{*}\in V(H_{i^{*}})} t_{u^{*}}(\overrightarrow{G}).
  \end{split}
\end{displaymath}

The Laplacian matrix of $\overrightarrow{G}'$ is $L_{\overrightarrow{G}'}=\left(\begin{matrix}
D_{1} &-B\\
-C &D_{2}\\
\end{matrix}
\right),$
where $D_{1}$ is a diagonal matrix  corresponding to $E(G)$ with diagonal entries $(D_{1})_{ee}=|Q_{e}^{(2)}|$ for $e\in E(G)$, $D_{2}=diag(\Lambda_{1},\Lambda_{2},...,\Lambda_{n})$ is a block diagonal matrix such that
 \begin{displaymath}
\begin{split}
 \Lambda_{i}=L_{H_{i}}+\sum_{e\in E(G),t(e)=i}|Q_{e}^{(2)}|I,
  \end{split}
\end{displaymath}
$B$ is a $|E(G)|\times \sum_{i=1}^{n}m_{i}$ matrix with entries $(B)_{ev}=1$ if $v\in V(H_{j})$ and $e=(i,j)\in E(G)$, and $(B)_{ev}=0$ otherwise, $C$ is a $ \sum_{i=1}^{n}m_{i} \times|E(G)|$ matrix with entries $(C)_{ue}=|Q_{e}^{(2)}|$ if $u\in V(H_{i})$ and $e=(i,j)\in E(G)$, and $(C)_{ue}=0$ otherwise.

Since $d_{i}^{+}(G)>0$ for each $i\in V(G)$,  $D_{2}$ is nonsingular.    The Schur complement of $D_{2}$ in $L_{\overrightarrow{G}'}$ is $S=D_{1}-BD_{2}^{-1}C$. Suppose that $h(f_{1})=t(f_{2})=j$ for $f_{1},f_{2}\in E(G)$ with $f_{2}=(j,q)$ and  the column (row) partition of the matrix $B$ ($C$) is in keeping with  the row (column) partition of $D=diag(\Lambda_{1},\Lambda_{2},...,\Lambda_{n})$. Note that $\Lambda_{j}\mathrm{j}_{m_{j}}=\sum_{e\in E(G),t(e)=j}|Q_{e}^{(2)}|\mathrm{j}_{m_{j}}$. By computation, we have
\begin{displaymath}
\begin{split}
  (BD_{2}^{-1}C)_{f_{1}f_{2}}&=|Q_{f_{2}}^{(2)}|\mathrm{j}_{m_{j}}^{\top}\Lambda_{j}^{-1}\mathrm{j}_{m_{j}}
  =\frac{m_{j}|Q_{f_{2}}^{(2)}|}{\sum_{e\in E(G),t(e)=j}|Q_{e}^{(2)}|}=\frac{m_{j}m_{q}}{\sum_{(j,\ell)\in E(G)}m_{\ell}}.
  \end{split}
\end{displaymath}
If $h(f_{1})\neq t(f_{2})$, then $(BD_{2}^{-1}C)_{f_{1}f_{2}}=0$.

Let $G^{*}$ be a weighted digraph of $G$ with the weight $\omega_{e}(G^{*})=\frac{m_{i}m_{j}}{\sum_{(i,\ell)\in E(G)}m_{\ell}}$ for each $e=(i,j) \in E(G)$.    For $h(f_{1})=t(f_{2})=j$ with $f_{1},f_{2}\in E(G)$ and $f_{2}=(j,q)$, we  assign weight $\omega_{f_{2}}(G^{*})$ to the edge $(f_{1},f_{2})$ in the weighted line digraph $\mathcal{L}(G^{*})$ of $G^{*}$. Then the weights of $\mathcal{L}(G^{*})$ are induced by indeterminates $\{\omega_{e}(G^{*})\}_{e\in E(G)}$ and  $S$ is the Laplacian matrix of the weighted line digraph $\mathcal{L}(G^{*})$. Note that $d_{i}(G^{*})=m_{i}$. By Proposition \ref{x9} and  Theorem \ref{xb} $(2)$, for any vertex  $e^{*}=(i^{*},j^{*})\in E(G)$ of $\overrightarrow{G}'$ with $d_{j^{*}}^{+}(G)>0$, we have
\begin{displaymath}
\begin{split}
 t_{e^{*}}(\overrightarrow{G}',\omega)=&det(D_{2})t_{e^{*}}(\mathcal{L}(G^{*}),\omega)\\
 =&\prod_{i\in V(G)} det(\Lambda_{i})\cdot\omega_{e^{*}}(G^{*})t_{i^{*}}(G^{*},\omega)d_{j^{*}}(G^{*})^{d_{j^{*}}^{-}(G)-2}\prod_{\substack{i\in V(G)\\ i\neq j^{*}}}d_{i}(G^{*})^{d_{i}^{-}(G)-1}\\
 =&\prod_{i\in V(G)}\left(\sum_{(i,j)\in E(G)}m_{j}\cdot\prod_{k=1}^{m_{i}-1}\left(\mu_{k}(H_{i})+\sum_{(i,j)\in E(G)}m_{j}\right)\right)\\
 &\cdot\frac{m_{i^{*}}}{\sum_{(i^{*},j)\in E(G)}m_{j}}\prod_{i\in V(G)}m_{i}^{d_{i}^{-}(G)-1}t_{i^{*}}(G^{*},\omega)\\
 =&\prod_{i\in V(G)}\left(\sum_{(i,j)\in E(G)}m_{j}\cdot\prod_{k=1}^{m_{i}-1}\left(\mu_{k}(H_{i})+\sum_{(i,j)\in E(G)}m_{j}\right)m_{i}^{d_{i}^{-}(G)-1}\right)\\
 &\cdot\frac{m_{i^{*}}}{\sum_{(i^{*},j)\in E(G)}m_{j}}
 \cdot\sum_{T\in \mathcal{T}_{i^{*}}(G)}\prod_{(i,j)\in E(T)}\frac{m_{i}m_{j}}{\sum_{(i,\ell)\in E(G)}m_{\ell}}.\\
 \end{split}
\end{displaymath}
 Since $\prod_{(i,j)\in E(T)}\frac{m_{i}}{\sum_{(i,\ell)\in E(G)}m_{\ell}}=\prod_{i\in V(G), i\neq i^{*}}\frac{m_{i}}{\sum_{(i,\ell)\in E(G)}m_{\ell}}$ for any $T\in \mathcal{T}_{i^{*}}$,  the above equation equals
\begin{displaymath}
\begin{split}
&\prod_{i\in V(G)}\left(\sum_{(i,j)\in E(G)}m_{j}\cdot\prod_{k=1}^{m_{i}-1}\left(\mu_{k}(H_{i})+\sum_{(i,j)\in E(G)}m_{j}\right)m_{i}^{d_{i}^{-}(G)-1}\right)\\
 &\cdot\prod_{i\in V(G)}\frac{m_{i}}{\sum_{(i,j)\in E(G)}m_{j}}
 \sum_{T\in \mathcal{T}_{i^{*}}(G)}\prod_{(i,j)\in E(T)}m_{j}\\
 =&\prod_{i\in V(G)}\left(\prod_{k=1}^{m_{i}-1}\left(\mu_{k}(H_{i})+\sum_{(i,j)\in E(G)}m_{j}\right)m_{i}^{d_{i}^{-}(G)}\right)\cdot \sum_{T\in \mathcal{T}_{i^{*}}(G)}\prod_{(i,j)\in E(T)}m_{j}.
 \end{split}
\end{displaymath}
  Note that $\prod_{e\in E(G)}|Q_{e}^{(2)}|=\prod_{(i,j)\in E(G)}m_{j}=\prod_{i\in V(G)}m_{i}^{d_{i}^{-}(G)}$. Hence, the result follows from Equation (4.1).

Finally, we consider that $d_{i'}^{+}(G)=0$ for some $i'\in V(G)$.   If $i'=i^{*}$, we have $t(\overrightarrow{G})=\sum_{u^{*}\in V(H_{i^{*}})}t_{u^{*}}(\overrightarrow{G}) $ and $\mathcal{T}(G)=\mathcal{T}_{i^{*}}(G)$, the result follows from Theorem \ref{x139}. If $i'\neq i^{*}$, we have $\sum_{u^{*}\in V(H_{i^{*}})}t_{u^{*}}(\overrightarrow{G})=0$ and $\mathcal{T}_{i^{*}}(G)=\emptyset$, the result also holds.
{\hfill}
\end{proof}

It is worth mentioning that we can also deduce Theorem \ref{x139} from Theorem \ref{x13} by $t(\overrightarrow{G})=\sum_{i^{*}\in V(G)}\sum_{u^{*}\in V(H_{i^{*}})}t_{u^{*}}(\overrightarrow{G})$ when  $d_{i}^{+}(G)>0$ for each $i\in V(G)$.

In order to obtain the enumerative formula of oriented spanning trees with a fixed root in the generalized join digraph $G[H_{1},H_{2},...,H_{n}]$ when $H_{1}, H_{2},...,H_{n}$ are  some special  graphs, we need the following preliminary result.

\begin{lemma}\label{xsh3}
Let $G$ be a digraph with $u^{*},v^{*}\in V(G)$ and   $ N_{G}^{-}(u^{*})\setminus \{v^{*}\}= N_{G}^{-}(v^{*})\setminus \{u^{*}\}$. Then
\begin{displaymath}
\left(d_{u^{*}}^{+}(G)+x_{u^{*}v^{*}}\right)t_{u^{*}}(G)
=
\left(d_{v^{*}}^{+}(G)+x_{v^{*}u^{*}}\right)t_{v^{*}}(G),
\end{displaymath}
where \begin{displaymath}
x_{uv}=
   \begin{cases}
   1 &\mbox {\rm if $(u,v)\in E(G)$}\\
 0  &\mbox {\rm if $(u,v)\notin E(G)$}
   \end{cases}
\end{displaymath}
for $u,v\in V(G)$.
\end{lemma}
\begin{proof}
Since  $u^{*},v^{*} \in V(G)$ satisfy $ N_{G}^{-}(u^{*})\setminus \{v^{*}\}= N_{G}^{-}(v^{*})\setminus \{u^{*}\}$, then
the Laplacian matrix of $G$ can be written as
$$L_{G}=\left(\begin{matrix}
d_{u^{*}}^{+}(G) &-x_{u^{*}v^{*}} &-\alpha^{\top}\\
-x_{v^{*}u^{*}} &d_{v^{*}}^{+}(G) &-\beta^{\top}\\
-\gamma &-\gamma &L_{1}\\
\end{matrix}
\right),$$
where $L_{1}$ is the principal submatrix of $L_{G}$ corresponding to $V(G)\setminus \{u^{*},v^{*}\}$, and $\alpha,\beta,\gamma$ are three column vectors such that
\begin{displaymath}
  \alpha_{u}= \begin{cases}
1 &\mbox{if  $u\in N_{G}^{+}(u^{*})\setminus \{v^{*}\}$},\\
0 &\mbox{otherwise};\\
    \end{cases}
\end{displaymath}

   \begin{displaymath}
   \beta_{u}= \begin{cases}
1 &\mbox{if  $u\in N_{G}^{+}(v^{*})\setminus\{u^{*}\}$},\\
0 &\mbox{otherwise};\\
   \end{cases}
\end{displaymath}

 \begin{displaymath}
   \gamma_{u}= \begin{cases}
1 &\mbox{if  $u\in N_{G}^{-}(u^{*})\setminus \{v^{*}\}=N_{G}^{-}(v^{*})\setminus \{u^{*}\}$},\\
0 &\mbox{otherwise}.\\
   \end{cases}
\end{displaymath}

Let $X$ be the adjoint matrix of $L_{G}$. Then
\begin{displaymath}\tag{4.2}
 XL_{G}=det(L_{G})I=0.
\end{displaymath}
By Lemma \ref{x7}, we get
\begin{displaymath}\tag{4.3}
 (X)_{uv}=t_{v}(G),~{\rm for~any~}u,v\in V(G).
\end{displaymath}
Hence,  by Equations (4.2) and (4.3),  we have
\begin{displaymath}\tag{4.4}
 \sum_{v\in V(G)}t_{v}(G)(L_{G})_{vu^{*}}=t_{u^{*}}(G)d_{u^{*}}^{+}(G)-x_{v^{*}u^{*}}t_{v^{*}}(G)-\sum_{v\in N_{G}^{-}(u^{*})\setminus \{v^{*}\}}t_{v}(G)=0
\end{displaymath}
and
\begin{displaymath}\tag{4.5}
\sum_{v\in V(G)}t_{v}(G)(L_{G})_{vv^{*}}=t_{v^{*}}(G)d_{v^{*}}^{+}(G)-x_{u^{*}v^{*}}t_{u^{*}}(G)-\sum_{v\in N_{G}^{-}(v^{*})\setminus \{u^{*}\}}t_{v}(G)=0.
\end{displaymath}

By Equations (4.4) and (4.5),  we have $t_{u^{*}}(G)d_{u^{*}}^{+}(G)-x_{v^{*}u^{*}}t_{v^{*}}(G)=t_{v^{*}}(G)d_{v^{*}}^{+}(G)-x_{u^{*}v^{*}}t_{u^{*}}(G)$, the result follows.
{\hfill}
\end{proof}

When each $H_{i}$ is an independent set  of order $m_{i}$, the enumerative formula of oriented spanning trees with a fixed root in the generalized join digraph $G[H_{1},H_{2},...,H_{n}]$  can be obtained as follows.

\begin{theorem}\label{x131}
Let $\overrightarrow{G}=G[H_{1},H_{2},...,H_{n}]$  be a  generalized join digraph, where  $H_{i}$ is an independent set of order $m_{i}$. For $u^{*}\in V(H_{i^{*}})$ and $i^{*}\in V(G)$, we have
\begin{displaymath}
\begin{split}
 t_{u^{*}}(\overrightarrow{G})
 =&\frac{1}{m_{i^{*}}}\prod_{i\in V(G)}\prod_{k=1}^{m_{i}-1}\sum_{(i,j)\in E(G)}m_{j}\sum_{T\in \mathcal{T}_{i^{*}}(G)}\prod_{(i,j)\in E(T)}m_{j}.
  \end{split}
\end{displaymath}
\end{theorem}
\begin{proof}
Since each $H_{i}$ is an independent set,  we have $\mu_{k}(H_{i})=0$ for $1\leq i\leq n$ and $1\leq k\leq m_{i}$. Note that for any $u^{*},v^{*}\in V(H_{i^{*}})$, we have     $N_{\overrightarrow{G}}^{-}(u^{*})\setminus \{v^{*}\}= N_{\overrightarrow{G}}^{-}(v^{*})\setminus \{u^{*}\}$, $N_{\overrightarrow{G}}^{+}(u^{*})\setminus \{v^{*}\}= N_{\overrightarrow{G}}^{+}(v^{*})\setminus \{u^{*}\}$.
The permutation interchanging $u^{*}$ and $v^{*}$ and fixing all other vertices is an automorphism of $\overrightarrow{G}$. Hence,
\begin{displaymath}
t_{u^{*}}(\overrightarrow{G})=t_{v^{*}}(\overrightarrow{G}).
\end{displaymath}
Then the result follows from Theorem \ref{x13}.
{\hfill}
\end{proof}


Let $G$ be a digraph with vertex set $V(G)=\{1,2,...,n\}$. Suppose that  the generalized join digraph $\overrightarrow{G}=G[H_{1},H_{2},...,H_{n}]$ has vertex set $V(H_{1})\cup V(H_{2})\cup \cdots \cup V(H_{n})$, where $|V(H_{i})|=m_{i}$ and $V({H_{i}})=\{v_{i}^{1},v_{i}^{2},...,v_{i}^{m_{i}}\}$ for $1\leq i\leq n$. Assume that each digraph $H_{i}$ consists of an oriented  matching and isolated vertices, that is, $H_{i}=a_{i}K_{2}\cup b_{i}K_{1}$ where $ a_{i},b_{i} $ are two nonnegative  integers. Set $|V(H_{i})|=m_{i}=2a_{i}+b_{i}$. Without loss of generality,  let $a_{i}K_{2}=\bigcup_{s=1}^{a_{i}}\{(v_{i}^{2s-1},v_{i}^{2s})\}$ and $b_{i}K_{1}=\{v_{i}^{2a_{i}+1},v_{i}^{2a_{i}+2},...,v_{i}^{2a_{i}+b_{i}}\}$ for $1\leq i\leq n$.

The notations above are kept in the following theorem. Now  we give the enumerative formula of oriented spanning trees with a fixed root  in the generalized join digraph $G[H_{1},H_{2},...,H_{n}]$  when $H_{i}=a_{i}K_{2}\cup b_{i}K_{1}$ for each $i\in V(G)$.

\begin{theorem}\label{x1312}
Let $\overrightarrow{G}=G[H_{1},H_{2},...,H_{n}]$  be a  generalized join digraph with $|V(H_{i})|=m_{i}$ for each $1\leq i\leq n$, where  $H_{i}=a_{i}K_{2}\cup b_{i}K_{1}$. Set $i^{*}\in V(G)$ and $z_{i}(\overrightarrow{G})=\sum_{(i,j)\in E(G)}m_{j}$ for each $i\in V(G)$.

{\em (1)} If  $1\leq p\leq 2a_{i^{*}}$  is odd, then for $v_{i^{*}}^{p}\in V(H_{i^{*}})$,
\begin{displaymath}
\begin{split}
 t_{v_{i^{*}}^{p}}(\overrightarrow{G})=&\frac{z_{i^{*}}(\overrightarrow{G})}{m_{i^{*}}+m_{i^{*}}z_{i^{*}}(\overrightarrow{G})}\prod_{i\in V(G)}z_{i}(\overrightarrow{G})^{m_{i}-a_{i}-1}\left(1+z_{i}(\overrightarrow{G})\right)^{a_{i}}\sum_{T\in \mathcal{T}_{i^{*}}(G)}\prod_{(i,j)\in E(T)}m_{j}.
  \end{split}
\end{displaymath}

{\em (2)} If  $1\leq p\leq 2a_{i^{*}}$  is even, then for $v_{i^{*}}^{p}\in V(H_{i^{*}})$,
\begin{displaymath}
\begin{split}
 t_{v_{i^{*}}^{p}}(\overrightarrow{G})=&\frac{2+z_{i^{*}}(\overrightarrow{G})}{m_{i^{*}}+m_{i^{*}}z_{i^{*}}(\overrightarrow{G})}\prod_{i\in V(G)}z_{i}(\overrightarrow{G})^{m_{i}-a_{i}-1}\left(1+z_{i}(\overrightarrow{G})\right)^{a_{i}}\sum_{T\in \mathcal{T}_{i^{*}}(G)}\prod_{(i,j)\in E(T)}m_{j}.
  \end{split}
\end{displaymath}

{\em (3)} If  $2a_{i^{*}}+1\leq p\leq m_{i^{*}}$, then for $v_{i^{*}}^{p}\in V(H_{i^{*}})$,
\begin{displaymath}
\begin{split}
 t_{v_{i^{*}}^{p}}(\overrightarrow{G})=&\frac{1}{m_{i^{*}}}\prod_{i\in V(G)}z_{i}(\overrightarrow{G})^{m_{i}-a_{i}-1}\left(1+z_{i}(\overrightarrow{G})\right)^{a_{i}}\sum_{T\in \mathcal{T}_{i^{*}}(G)}\prod_{(i,j)\in E(T)}m_{j}.
  \end{split}
\end{displaymath}
\end{theorem}
\begin{proof}
Since each $H_{i}=a_{i}K_{2}\cup b_{i}K_{1}$,  we have $Spec(H_{i})=\{1^{[a_{i}]},0^{[m_{i}-a_{i}]}\}$ for $1\leq i\leq n$, where $\lambda^{[k]}$ means that $\lambda$ is an eigenvalue with algebraic multiplicity $k$.

If $a_{i^{*}}=0$, then all vertices of $H_{i^{*}}$ are symmetric, and assertion {\rm (3)} follows from Theorem \ref{x13}. Hence, assume that $a_{i^{*}}>0$. If $z_{i^{*}}(\overrightarrow{G})=0$, then the result is immediate: when $(a_{i^{*}},b_{i^{*}})=(1,0)$ only the terminal vertex can be a root, and otherwise $t_v(\overrightarrow{G})=0$ for every $v\in V(H_{i^{*}})$. Hence, assume that $z_{i^{*}}(\overrightarrow{G})>0$.

By symmetry, the numbers of oriented spanning trees with roots at the isolated vertices of $H_{i^{*}}$ are equal, and the numbers of oriented spanning trees with roots at the initial vertices of its oriented matching are equal. In particular, for $1\leq p,q\leq 2a_{i^{*}}$ with $p,q$ odd,
\begin{displaymath}\tag{4.6}
t_{v_{i^{*}}^{p}}(\overrightarrow{G})=t_{v_{i^{*}}^{q}}(\overrightarrow{G}).
\end{displaymath}

Similarly, by Lemma \ref{xsh3}, for any $v_{i^{*}}^{p},v_{i^{*}}^{q}\in V(H_{i^{*}})$ with $1\leq p\leq 2a_{i^{*}}$ ($p$ odd) and $2a_{i^{*}}+1\leq q\leq m_{i^{*}}$, we have
\begin{displaymath}\tag{4.7}
\left(1+z_{i^{*}}(\overrightarrow{G})\right)t_{v_{i^{*}}^{p}}(\overrightarrow{G})
=z_{i^{*}}(\overrightarrow{G})t_{v_{i^{*}}^{q}}(\overrightarrow{G}),
\end{displaymath}
and, if $(v_{i^{*}}^{p},v_{i^{*}}^{q})\in E(\overrightarrow{G})$ with $p$ odd and $q$ even, then
\begin{displaymath}\tag{4.8}
\left(2+z_{i^{*}}(\overrightarrow{G})\right)t_{v_{i^{*}}^{p}}(\overrightarrow{G})
=z_{i^{*}}(\overrightarrow{G})t_{v_{i^{*}}^{q}}(\overrightarrow{G}).
\end{displaymath}
By symmetry, the numbers of oriented spanning trees with roots at the terminal vertices of the oriented matching are also equal.

If $b_{i^{*}}=0$, let $t_{1}$ and $t_{2}$ denote the common values of $t_v(\overrightarrow{G})$ for the initial and terminal vertices, respectively. By Equation (4.8) and Theorem \ref{x13},
\[
a_{i^{*}}(t_{1}+t_{2})=
\frac{m_{i^{*}}+m_{i^{*}}z_{i^{*}}(\overrightarrow{G})}{z_{i^{*}}(\overrightarrow{G})}t_{1},
\]
so assertions {\rm (1)} and {\rm (2)} follow, while assertion {\rm (3)} is void. Hence, assume that $b_{i^{*}}>0$.

Let $t_{1}=t_{v_{i^{*}}^{p}}(\overrightarrow{G})$ for $1\leq p\leq 2a_{i^{*}}$ and $p$ is odd, $t_{2}=t_{v_{i^{*}}^{p}}(\overrightarrow{G})$ for $1\leq p\leq 2a_{i^{*}}$ and $p$ is even, and $t_{3}=t_{v_{i^{*}}^{p}}(\overrightarrow{G})$ for $2a_{i^{*}}+1\leq p\leq m_{i^{*}}$. By Equations (4.7) and (4.8),
we obtain
$t_{2}=\frac{2+z_{i^{*}}(\overrightarrow{G})}{z_{i^{*}}(\overrightarrow{G})}t_{1}$ and $t_{3}=\frac{1+z_{i^{*}}(\overrightarrow{G})}{z_{i^{*}}(\overrightarrow{G})}t_{1}$. By Theorem \ref{x13},  we have
\begin{displaymath}
\begin{split}
\sum_{v_{i^{*}}^{p}\in V(H_{i^{*}})} t_{v_{i^{*}}^{p}}(\overrightarrow{G})=&a_{i^{*}}t_{1}+a_{i^{*}}t_{2}+(m_{i^{*}}-2a_{i^{*}})t_{3}
=\frac{m_{i^{*}}+m_{i^{*}}z_{i^{*}}(\overrightarrow{G})}{z_{i^{*}}(\overrightarrow{G})}t_{1}\\
 =&\prod_{i\in V(G)}z_{i}(\overrightarrow{G})^{m_{i}-a_{i}-1}\left(1+z_{i}(\overrightarrow{G})\right)^{a_{i}}\sum_{T\in \mathcal{T}_{i^{*}}(G)}\prod_{(i,j)\in E(T)}m_{j}.
  \end{split}
\end{displaymath}
Thus the result follows.
{\hfill}
\end{proof}

Let $\overrightarrow{G}=G[H_{1},H_{2},...,H_{n}]$  be a  generalized join digraph with $|V(H_{i})|=m_{i}$ for $1\leq i\leq n$.
 Assume that each digraph $H_{i}$ consists of an oriented  star,  that is, $H_{i}=S_{m_{i}}$ in which $v_{i}^{1}$ is the central vertex of $S_{m_{i}}$ and $v_{i}^{2},v_{i}^{3},...,v_{i}^{m_{i}}$ are  pendant vertices of $S_{m_{i}}$ such that $(v_{i}^{p},v_{i}^{1})\in E(S_{m_{i}})$ for each $2\leq p\leq m_{i}$.

The notations above are kept in the following theorem. Now  we give the enumerative formula of oriented spanning trees with a fixed root  in the generalized join digraph $G[S_{m_1},S_{m_2},...,S_{m_n}]$.

\begin{theorem}\label{x1313}
Let $\overrightarrow{G}=G[S_{m_1},S_{m_2},...,S_{m_n}]$  be a  generalized join digraph, where  $S_{m_i}$ is   an oriented star of order $m_{i}$ for $1\leq i\leq n$. Set $i^{*}\in V(G)$ and $z_{i}(\overrightarrow{G})=\sum_{(i,j)\in E(G)}m_{j}$ for each $i\in V(G)$.

{\em (1)} If $p=1$, then for $v_{i^{*}}^{1}\in V(S_{m_{i^{*}}})$,
\begin{displaymath}
\begin{split}
 t_{v_{i^{*}}^{1}}(\overrightarrow{G})=&\frac{m_{i^{*}}+z_{i^{*}}(\overrightarrow{G})}{m_{i^{*}}+m_{i^{*}}z_{i^{*}}(\overrightarrow{G})}\prod_{i\in V(G)}\left(1+z_{i}(\overrightarrow{G})\right)^{(m_{i}-1)}\sum_{T\in \mathcal{T}_{i^{*}}(G)}\prod_{(i,j)\in E(T)}m_{j}.
  \end{split}
\end{displaymath}

{\em (2)} If  $2\leq p\leq m_{i^{*}}$,  then for $v_{i^{*}}^{p}\in V(S_{m_{i^{*}}})$,
\begin{displaymath}
\begin{split}
 t_{v_{i^{*}}^{p}}(\overrightarrow{G})=&\frac{z_{i^{*}}(\overrightarrow{G})}{m_{i^{*}}+m_{i^{*}}z_{i^{*}}(\overrightarrow{G})}\prod_{i\in V(G)}\left(1+z_{i}(\overrightarrow{G})\right)^{(m_{i}-1)}\sum_{T\in \mathcal{T}_{i^{*}}(G)}\prod_{(i,j)\in E(T)}m_{j}.
  \end{split}
\end{displaymath}
\end{theorem}
\begin{proof}
By Lemma \ref{xsh3},  for any $v_{i^{*}}^{p},v_{i^{*}}^{q}\in V(S_{m_{i^{*}}})$ and $2\leq p,q\leq m_{i^{*}}$, we have
\begin{displaymath}
\left(1+z_{i^{*}}(\overrightarrow{G})\right)t_{v_{i^{*}}^{p}}(\overrightarrow{G})
=
\left(1+z_{i^{*}}(\overrightarrow{G})\right)t_{v_{i^{*}}^{q}}(\overrightarrow{G}).
\end{displaymath}
Hence, $t_{v_{i^{*}}^{p}}(\overrightarrow{G})=t_{v_{i^{*}}^{q}}(\overrightarrow{G})$ for any $v_{i^{*}}^{p},v_{i^{*}}^{q}\in V(S_{m_{i^{*}}})$ and $2\leq p,q\leq m_{i^{*}}$.

 Note that $v_{i^{*}}^{1},v_{i^{*}}^{p}\in V(S_{m_{i^{*}}})$ with $2\leq p\leq m_{i^{*}}$ satisfy $ N_{\overrightarrow{G}}^{-}(v_{i^{*}}^{1})\setminus \{v_{i^{*}}^{2},v_{i^{*}}^{3},...,v_{i^{*}}^{m_{i^{*}}}\}= N_{\overrightarrow{G}}^{-}(v_{i^{*}}^{p})$, then
the Laplacian matrix of $\overrightarrow{G}$ can be written as
$$L_{\overrightarrow{G}}=\left(\begin{matrix}
d_{v_{i^{*}}^{1}}^{+}(\overrightarrow{G}) &0 &\mathbf{0}_{m_{i^{*}}-2}^{\top} &-\alpha^{\top}\\
-1 &d_{v_{i^{*}}^{p}}^{+}(\overrightarrow{G}) &\mathbf{0}_{m_{i^{*}}-2}^{\top} &-\beta^{\top}\\
-\mathrm{j}_{m_{i^{*}}-2} &\mathbf{0}_{m_{i^{*}}-2} &d_{v_{i^{*}}^{p}}^{+}(\overrightarrow{G})I_{m_{i^{*}}-2} &-\mathrm{j}_{m_{i^{*}}-2}\beta^{\top}\\
-\gamma &-\gamma &-\gamma\mathrm{j}_{m_{i^{*}}-2}^{\top} &L_{1}\\
\end{matrix}
\right),$$
where $L_{1}$ is the principal submatrix of $L_{\overrightarrow{G}}$ corresponding to $V(\overrightarrow{G})\setminus \{v_{i^{*}}^{1},v_{i^{*}}^{2},...,v_{i^{*}}^{m_{i^{*}}}\}$, $I_{m_{i^{*}}-2}$ is the identity matrix of order $m_{i^{*}}-2$, $\mathbf{0}_{m_{i^{*}}-2}$ is the all-zero column vector of dimension $m_{i^{*}}-2$, and $\alpha,\beta,\gamma$ are three column vectors indexed by $V(\overrightarrow{G})\setminus \{v_{i^{*}}^{1},v_{i^{*}}^{2},...,v_{i^{*}}^{m_{i^{*}}}\}$ such that, for $v\in V(\overrightarrow{G})\setminus \{v_{i^{*}}^{1},v_{i^{*}}^{2},...,v_{i^{*}}^{m_{i^{*}}}\}$,
\begin{displaymath}
  \alpha_{v}= \begin{cases}
1 &\mbox{if  $v\in N_{\overrightarrow{G}}^{+}(v_{i^{*}}^{1})$},\\
0 &\mbox{otherwise};\\
    \end{cases}
\end{displaymath}

   \begin{displaymath}
   \beta_{v}= \begin{cases}
1 &\mbox{if  $v\in N_{\overrightarrow{G}}^{+}(v_{i^{*}}^{p})\setminus\{v_{i^{*}}^{1}\}$},\\
0 &\mbox{otherwise};\\
   \end{cases}
\end{displaymath}

 \begin{displaymath}
   \gamma_{v}= \begin{cases}
1 &\mbox{if  $v\in N_{\overrightarrow{G}}^{-}(v_{i^{*}}^{1})\setminus \{v_{i^{*}}^{2},v_{i^{*}}^{3},...,v_{i^{*}}^{m_{i^{*}}}\}=N_{\overrightarrow{G}}^{-}(v_{i^{*}}^{p})$},\\
0 &\mbox{otherwise}.\\
   \end{cases}
\end{displaymath}

Let $X$ be the adjoint matrix of $L_{\overrightarrow{G}}$. Then
\begin{displaymath}\tag{4.9}
 XL_{\overrightarrow{G}}=det(L_{\overrightarrow{G}})I=0.
\end{displaymath}
By Lemma \ref{x7}, we get
\begin{displaymath}\tag{4.10}
 (X)_{uv}=t_{v}(\overrightarrow{G}),~{\rm for~any~}u,v\in V(\overrightarrow{G}).
\end{displaymath}
Hence,  by Equations (4.9) and (4.10), for  $v_{i^{*}}^{p}\in V(S_{m_{i^{*}}})$ with $2\leq p\leq m_{i^{*}}$, we have
\begin{displaymath}\tag{4.11}
 \sum_{v\in V(\overrightarrow{G})}t_{v}(\overrightarrow{G})(L_{\overrightarrow{G}})_{vv_{i^{*}}^{1}}=t_{v_{i^{*}}^{1}}(\overrightarrow{G})d_{v_{i^{*}}^{1}}^{+}(\overrightarrow{G})-(m_{i^{*}}-1)t_{v_{i^{*}}^{p}}(\overrightarrow{G})-\sum_{v\in N_{\overrightarrow{G}}^{-}(v_{i^{*}}^{p})}t_{v}(\overrightarrow{G})=0
\end{displaymath}
and
\begin{displaymath}\tag{4.12}
\sum_{v\in V(\overrightarrow{G})}t_{v}(\overrightarrow{G})(L_{\overrightarrow{G}})_{vv_{i^{*}}^{p}}=t_{v_{i^{*}}^{p}}(\overrightarrow{G})d_{v_{i^{*}}^{p}}^{+}(\overrightarrow{G})-\sum_{v\in N_{\overrightarrow{G}}^{-}(v_{i^{*}}^{p})}t_{v}(\overrightarrow{G})=0.
\end{displaymath}
By Equations (4.11) and (4.12), we have
\begin{displaymath}\tag{4.13}
\left(z_{i^{*}}(\overrightarrow{G})+m_{i^{*}}\right)t_{v_{i^{*}}^{p}}(\overrightarrow{G})
=
z_{i^{*}}(\overrightarrow{G})t_{v_{i^{*}}^{1}}(\overrightarrow{G}).
\end{displaymath}

Let $t_{1}=t_{v_{i^{*}}^{1}}(\overrightarrow{G})$ and $t_{2}=t_{v_{i^{*}}^{p}}(\overrightarrow{G})$ for $2\leq p\leq m_{i^{*}}$. By Equation (4.13), we get
$t_{2}=\frac{z_{i^{*}}(\overrightarrow{G})}{z_{i^{*}}(\overrightarrow{G})+m_{i^{*}}}t_{1}$. Note that $Spec(S_{m_{i}})=\{1^{[m_{i}-1]},0\}$ for $1\leq i\leq n$. By Theorem \ref{x13},  we have
\begin{displaymath}
\begin{split}
\sum_{v_{i^{*}}^{q}\in V(S_{m_{i^{*}}})} t_{v_{i^{*}}^{q}}(\overrightarrow{G})=&t_{1}+(m_{i^{*}}-1)t_{2}
=\frac{m_{i^{*}}+m_{i^{*}}z_{i^{*}}(\overrightarrow{G})}{m_{i^{*}}+z_{i^{*}}(\overrightarrow{G})}t_{1}\\
 =&\prod_{i\in V(G)}\left(1+z_{i}(\overrightarrow{G})\right)^{(m_{i}-1)}\sum_{T\in \mathcal{T}_{i^{*}}(G)}\prod_{(i,j)\in E(T)}m_{j}.
  \end{split}
\end{displaymath}
Thus the result follows.
{\hfill}
\end{proof}

 If each $|V(H_{i})|=m_{i}=m$ for $1\leq i\leq n$, then  $z_{i}(\overrightarrow{G})=\sum_{(i,j)\in E(G)}m_{j}=md_{i}^{+}(G)$. By Theorems \ref{x131}, \ref{x1312} and \ref{x1313}, we obtain the following corollaries.
\begin{corollary}\label{x1314}
Let $\overrightarrow{G}=G[H_{1},H_{2},...,H_{n}]$  be a  generalized join digraph, where  $H_{i}$ is an independent set of order $m$. For $u^{*}\in V(H_{i^{*}})$ and $i^{*}\in V(G)$, we have
\begin{displaymath}
\begin{split}
 t_{u^{*}}(\overrightarrow{G})
 =&m^{mn-2}\prod_{i\in V(G)}d_{i}^{+}(G)^{m-1}t_{i^{*}}(G).
  \end{split}
\end{displaymath}
\end{corollary}

\begin{corollary}\label{x1315}
Let $\overrightarrow{G}=G[H_{1},H_{2},...,H_{n}]$  be a  generalized join digraph with $|V(H_{i})|=m$ for each $1\leq i\leq n$, where  $H_{i}=a_{i}K_{2}\cup b_{i}K_{1}$. Set $i^{*}\in V(G)$.

{\em (1)} If  $1\leq p\leq 2a_{i^{*}}$  is odd, then for $v_{i^{*}}^{p}\in V(H_{i^{*}})$,
\begin{displaymath}
\begin{split}
 t_{v_{i^{*}}^{p}}(\overrightarrow{G})=&\frac{m^{n-1}d_{i^{*}}^{+}(G)}{1+md_{i^{*}}^{+}(G)}\prod_{i\in V(G)}(md_{i}^{+}(G))^{m-a_{i}-1}\left(1+md_{i}^{+}(G)\right)^{a_{i}}t_{i^{*}}(G).
  \end{split}
\end{displaymath}

{\em (2)} If  $1\leq p\leq 2a_{i^{*}}$   is even, then for $v_{i^{*}}^{p}\in V(H_{i^{*}})$,
\begin{displaymath}
\begin{split}
 t_{v_{i^{*}}^{p}}(\overrightarrow{G})=&\frac{m^{n-2}(2+md_{i^{*}}^{+}(G))}{1+md_{i^{*}}^{+}(G)}\prod_{i\in V(G)}(md_{i}^{+}(G))^{m-a_{i}-1}\left(1+md_{i}^{+}(G)\right)^{a_{i}}t_{i^{*}}(G).
  \end{split}
\end{displaymath}

{\em (3)} If  $2a_{i^{*}}+1\leq p\leq m$, then for $v_{i^{*}}^{p}\in V(H_{i^{*}})$,
\begin{displaymath}
\begin{split}
 t_{v_{i^{*}}^{p}}(\overrightarrow{G})=m^{n-2}\prod_{i\in V(G)}(md_{i}^{+}(G))^{m-a_{i}-1}\left(1+md_{i}^{+}(G)\right)^{a_{i}}t_{i^{*}}(G).
  \end{split}
\end{displaymath}
\end{corollary}

\begin{corollary}\label{x1316}
Let $\overrightarrow{G}=G[H_{1},H_{2},...,H_{n}]$  be a  generalized join digraph, where each $H_{i}$ is   an oriented star $S_{m}$. Set $i^{*}\in V(G)$.

{\em (1)} If $p=1$, then for $v_{i^{*}}^{1}\in V(H_{i^{*}})$,
\begin{displaymath}
\begin{split}
 t_{v_{i^{*}}^{1}}(\overrightarrow{G})=&\frac{m^{n-1}(1+d_{i^{*}}^{+}(G))}{1+md_{i^{*}}^{+}(G)}\prod_{i\in V(G)}\left(1+md_{i}^{+}(G)\right)^{(m-1)}t_{i^{*}}(G).
  \end{split}
\end{displaymath}

{\em (2)} If  $2\leq p\leq m$,  then for $v_{i^{*}}^{p}\in V(H_{i^{*}})$,
\begin{displaymath}
\begin{split}
 t_{v_{i^{*}}^{p}}(\overrightarrow{G})=&\frac{m^{n-1}d_{i^{*}}^{+}(G)}{1+md_{i^{*}}^{+}(G)}\prod_{i\in V(G)}\left(1+md_{i}^{+}(G)\right)^{(m-1)}t_{i^{*}}(G).
  \end{split}
\end{displaymath}
\end{corollary}

\section*{Data Availability Statement}
No data is used in this research.

\section*{Conflict of Interest Statement}
The authors have no conflict of interest.


\end{document}